\documentclass[reqno, oneside, 11pt]{amsart}
\usepackage[letterpaper]{geometry}
\usepackage{appendix}
\usepackage{multirow}
\usepackage{amsmath}
\usepackage{amssymb}
\usepackage{enumerate,enumitem}
\usepackage{verbatim}
\usepackage{mathrsfs}
\usepackage{mathtools}
\usepackage{hyperref}
\usepackage{amsthm, mathdots, multicol,placeins}

\usepackage{xcolor}
\usepackage{ dsfont }

\usepackage{thmtools}
\usepackage{thm-restate}

\usepackage{tikz}
\usetikzlibrary{matrix,fit}
\pgfkeys{tikz/mymatrix/.style={matrix of math nodes,inner sep=0pt,row sep=0em,column sep=0em,nodes={inner sep=5pt}}}
\newcommand*\mymatrixbox[5][]{\node [fit= (m-#2-#3) (m-#4-#5)] [draw=black,thick,inner sep=.5pt,#1] {};}

\newcommand{\vertsimeq}{\rotatebox{270}{$\,\simeq$}}
\newcommand{\simequal}[2]{\underset{\scriptstyle\overset{\mkern4mu\vertsimeq}{#2}}{#1}}

\DeclareMathOperator{\Norm}{Norm}
\DeclareMathOperator{\Jac}{Jac}
\DeclareMathOperator{\USp}{USp}
\DeclareMathOperator{\SU}{SU}
\DeclareMathOperator{\Sp}{Sp}
\DeclareMathOperator{\GSp}{GSp}
\DeclareMathOperator{\U}{U}
\DeclareMathOperator{\GL}{GL}
\DeclareMathOperator{\SL}{SL}

\DeclareMathOperator{\ST}{ST}
\DeclareMathOperator{\Gal}{Gal}

\DeclareMathOperator{\End}{End}
\DeclareMathOperator{\Aut}{Aut}
\DeclareMathOperator{\diag}{diag}
\DeclareMathOperator{\antidiag}{antidiag}
\DeclareMathOperator{\TL}{TL}

\DeclareMathOperator{\AST}{AST}
\DeclareMathOperator{\Hg}{Hg}

\DeclareMathOperator{\LL}{L}
\DeclareMathOperator{\tr}{tr}

\DeclareMathOperator{\MT}{MT}

\newcommand{\Z}{\mathbb{Z}}
\newcommand{\Q}{\mathbb{Q}}
\newcommand{\C}{\mathbb{C}}

\newtheorem{theorem}{Theorem}[section]
\newtheorem{example}[theorem]{Example}

\newtheorem{proposition}[theorem]{Proposition}

\newtheorem{lemma}[theorem]{Lemma}

\newtheorem{corollary}[theorem]{Corollary}
\newtheorem{conjecture}[theorem]{Conjecture}

\newtheorem{remark}[theorem]{Remark}

\theoremstyle{definition}

\theoremstyle{remark}
\newtheorem*{remark*}{Remark}

\author{Melissa Emory }
\address{Department of Mathematics, Oklahoma State University; 401 Mathematical Sciences, 
Oklahoma State University
Stillwater, OK 74078}
\email{melissa.emory@okstate.edu}

\author{Heidi Goodson}
\address{Department of Mathematics, Brooklyn College, City University of New York; 2900 Bedford Avenue, Brooklyn, NY 11210 USA}
\email{heidi.goodson@brooklyn.cuny.edu}

\title[Nondegeneracy and Sato-Tate Distributions]{Nondegeneracy and Sato-Tate Distributions of Two Families of Jacobian Varieties}

\subjclass[2010]{11M50, 11G10, 11G20, 14G10} 
\keywords{Hodge groups, Nondegeneracy, Sato-Tate groups, Sato-Tate distributions, Hyperelliptic curves}

\begin{document}
\maketitle 

\begin{abstract}
We consider the curves 
$ y^2=x^{2^m} -1$ and $y^2=x^{2^{d}+1}-x$ over the rationals.  These curves are related via their associated Jacobian varieties in that the Jacobians of the latter appear as factors of the Jacobians of the former. One of the principal aims of this paper is to fully describe their Sato-Tate groups and distributions by determining generators of the component groups. In order to do this, we first prove the nondegeneracy of the two families of Jacobian varieties via their Hodge groups. We then use results relating Sato-Tate groups and twisted Lefschetz groups of nondegenerate abelian varieties to determine the generators of the associated Sato-Tate groups. The results of this paper add new examples to the literature of families of nondegenerate Jacobian varieties and of noncyclic component groups of Sato-Tate groups. Furthermore, we compute moment statistics associated to the Sato-Tate groups which can be used to verify the equidistribution statement of the generalized Sato-Tate conjecture by comparing them to moment statistics obtained for the traces in the normalized $L$-polynomials of the curves. 
\end{abstract}

\section{Introduction}

In this paper we consider two families of hyperelliptic curves over the rationals
$$C_{2^m}: y^2=x^{2^m} -1, \hspace{.5 in} C'_{2^{d}+1}: y^2=x^{2^{d}+1}-x.$$ These curves are related via their associated Jacobian varieties: the Jacobians of $C'_{2^{d}+1}$ appear as factors of the Jacobians of $C_{2^m}$ for $m\geq 3$ and $1 \leq d \leq m-1$. The authors first studied curves of this form in \cite{EmoryGoodsonPeyrot} where we developed an algorithm to compute the identity components of the Sato-Tate groups of their Jacobians. One of the principal aims of this paper is to fully describe their Sato-Tate groups and distributions by determining generators of the component groups. The techniques that were used in \cite{EmoryGoodson2022, GoodsonCatalan} apply to curves whose Jacobians are nondegenerate, and so another principal aim of this paper is to prove the nondegeneracy of the two families of Jacobian varieties. Thus, this paper adds new examples to the literature of infinite families of nondegenerate Jacobian varieties and of interesting Sato-Tate groups.

A complex abelian variety $A$ is said to be \emph{stably nondegenerate} if the (complexified) Hodge ring of $A^n$ for $n\geq 1$ is generated by divisor classes. While the definition of nondegeneracy is a statement about the Hodge ring, we can see the effects of nondegeneracy in  groups constructed from the Hodge structure of the abelian variety: the Mumford-Tate group, the Hodge group, and the Sato-Tate group. The Mumford-Tate group and Hodge group are related to the Hodge ring via an action on certain cohomology groups. The algebraic Sato-Tate conjecture and Mumford-Tate conjecture imply a relationship between the identity component of the Sato-Tate group and the Mumford-Tate and Hodge groups. Thus, it is natural that the stable nondegeneracy of the Hodge ring has implications for these groups. In particular, we know that the Hodge group is maximal and so, by \cite[Theorem 2.16a]{Fite2012}, the algebraic Sato-Tate group is isomorphic to the twisted Lefschetz group (see Proposition \ref{prop:AST=TL}).  Furthermore, it follows from \cite[Theorem 6.1]{Banaszak2015} and work of Serre in \cite[Section 8.3.4]{SerreNXP} that if $A$ is a stably nondegenerate abelian variety defined over a number field $F$, then the component group $\ST(A)/\ST^0(A)$ is isomorphic to the Galois group $\Gal(K/F)$, where $K$ is the endomorphism field (i.e., the minimal extension over which all endomorphisms of the abelian variety are defined). It is for these reasons that it is particularly nice to work with stably nondegenerate abelian varieties in the context of Sato-Tate distributions, and why it is of interest to find new families of stably nondegenerate abelian varieties.

The study of Sato-Tate groups began with the original Sato-Tate conjecture for elliptic curves posed around 1960 by Mikio Sato and John Tate (independently). It is a statistical conjecture regarding the distribution of the normalized traces of Frobenius on an elliptic curve.  In 2012, the conjecture was generalized to higher dimensional abelian varieties by Serre \cite{SerreNXP}, and determining Sato-Tate groups of abelian varieties is the source of ongoing interest and work. Following the exposition of \cite{fitekedlayasuth2023} we discuss the generalized Sato-Tate conjecture.  
Let $A$ be a $g$-dimensional abelian variety over a number field $F$. Up to finitely many factors, corresponding to the primes of bad reduction, the $L$-function associated to the abelian variety $A$ is 
$$L(A,s)=\displaystyle\prod_{\mathfrak{p}}L_{\mathfrak{p}}(\Norm(\mathfrak{p})^{-s})^{-1},$$
where we take the product over prime ideals $\mathfrak{p}$ of the ring of integers of $F$ at which $A$ has good reduction\footnote{In general, when the dimension of $A>1$ determining the Euler factors at bad primes is difficult except in certain cases.   When $A$ is the Jacobian of a hyperelliptic curve, practical methods are known, and we refer the reader to \cite{DS1991} and \cite{Silverman}.}, the polynomial $L_{\mathfrak{p}}(T)$ is the characteristic polynomial associated to the reverse Frobenius acting on the $\ell$-adic Tate module when $\ell$ is coprime to $\mathfrak{p}$.  Normalizing $L_{\mathfrak{p}}$ so that the roots of the polynomial lie on the circle centered at the origin with norm 1, one can ask whether or not the resulting polynomials are equidistributed with respect to some measure.

 One can define a certain compact Lie subgroup of $\USp(g)$ associated to $A$ and a sequence of conjugacy classes therein whose characteristic polynomials are the normalized $L$-polynomials, which are equidistributed for the image of the Haar measure.  In  \cite{SerreNXP} Serre develops a set of axioms that this group (and the conjugacy classes therein) is expected to satisfy; we call this the {\it Sato-Tate group} of $A$ (see also \cite{SutherlandNotes}).  It is natural then to conjecture a corresponding equidistribution statement.  

We now specify to the case where $A$ is the Jacobian of a smooth projective curve $C$ and we now state the generalized Sato-Tate conjecture taken from \cite[Conjecture 3.6]{SutherlandNotes}.
\begin{conjecture}\label{conjec:generalST}
(Generalized Sato-Tate Conjecture)
    Let $\Jac(C_F)$ be the Jacobian of a smooth projective curve $C$ over a number field $F$, denote its Sato-Tate group $\ST(\Jac(C_F))$ and let $X_\mathfrak{p}$ be the sequence of conjugacy classes of normalized images of Frobenius elements in $\ST(\Jac(C_F))$ at primes $\mathfrak{p}$ of good reduction for $\Jac(C)_F$, ordered by norm (break ties arbitrarily). Then the sequence $X_\mathfrak{p}$ is equidistributed (with respect to the pushforward of the Haar measure of $\ST(\Jac(C_F))$ to its space of conjugacy classes).
\end{conjecture}

This equidistribution statement can be considered as an analogue of the Chebotaryov density theorem (which describes statistically the splitting of primes in a given Galois extension).  In this context the Sato-Tate group functions as the Galois group, and we then consider the Galois group associated to an arbitrary motive as in \cite{Serre1994}. Since the Sato-Tate group of a motive is a compact Lie group, we can retain much (but not all) of the same information by keeping track of the identity component and the group of connected components.

Although the generalized Sato-Tate conjecture is not fully proven, much is known for low dimensional cases. When $A$ is an elliptic curve with complex multiplication (CM), the distribution is precisely one of two depending on whether the CM field is contained in $F$ or not (see \cite{SutherlandNotes} for an exposition on this).  If the elliptic curve does not admit CM, the Sato-Tate conjecture predicts that exactly one distribution occurs and this has been proven over totally real fields (see \cite{Barnet2011, Clozel2008,Harris2010,Taylor2008}). Classification results for dimension 2 and 3 abelian varieties are given in \cite{Fite2012} and \cite{fitekedlayasuth2023}, respectively. The Sato-Tate conjecture was proven for higher dimensional CM abelian varieties in \cite{Joh17}, and there has been recent progress on computing Sato-Tate groups of nondegenerate abelian varieties (see, for example, \cite{Arora2016, EmoryGoodson2022,FGL2016, FiteLorenzoSutherland2018, GoodsonCatalan,GoodsonHoque2024,KedlayaSutherland2009, LarioSomoza2018}). The main goals of this paper are to prove that the Jacobians of the curves $C_{2^m}$ and $C'_{2^d+1}$ are nondegenerate and to determine their respective Sato-Tate groups. These curves are of particular interest since the Galois groups of their CM fields over $\Q$ are non-cyclic.

 Our first main results of the paper are regarding the nondegeneracy of the Jacobian varieties. For example, we have the following result (see Theorems \ref{theorem:nondegeneracyx^n-x} and \ref{thm:nondegeneracyx^2m-1}, and Corollary \ref{cor:nondegeneracyboth}). 
\begin{theorem}\label{thm:nondegneracyresult}
    Let $m\geq3$ and $d\geq 1$ be integers. Then the Jacobians of the curves $C_{2^m}: y^2=x^{2^m} -1$ and $C'_{2^{d}+1}: y^2=x^{2^{d}+1}-x$ and their twists are nondegenerate.
\end{theorem}

The proof of this result relies on the decomposition of the Jacobian varieties $\Jac(C_{2^m})$ given in  Section \ref{sec:decomposition}   and on analyzing the Hodge groups of the factors $\Jac(C'_{2^d+1})$. This result is highly dependent on the degree of $x$. For example, in \cite{GoodsonDegeneracy}, the second author proved that $\Jac(y^2=x^n-1)$ is degenerate whenever $n$ is divisible by $pq$, where $p$ and $q$ are (not necessarily distinct) odd primes. Until recently, there was not a technique in the literature for computing the Sato-Tate groups of degenerate abelian varieties, however one has now been established specifically for the family $\Jac(y^2=x^n+1)$ in \cite{GalleseGoodsonLombardo}.

Once we have established  nondegeneracy, we are able to give a complete description of the Sato-Tate group for each family of curves by first determining their twisted Lefschetz groups. We use similar techniques to those used in \cite{EmoryGoodson2022} and \cite{GoodsonCatalan}. The Jacobians of the curves $C'_{2^{d}+1}$ appear as factors of the Jacobians of the curves $C_{2^m}$, and we are able to describe the Sato-Tate group of the latter in terms of that of the former.

\begin{restatable*}{theorem}{theoremST}
{\label{thm:STx^2m-1}}
Let $g$ be the genus of the curve $C_{2^m}: y^2=x^{2^m}-1$. Let $\gamma_J$ denote the $2g\times2g$ diagonal block matrix $\diag(J,J,\ldots, J)$, and let $\gamma$ be the $2g\times2g$ diagonal block matrix whose diagonal entries are the matrices $\gamma_d$ in Theorem \ref{thm:STx^n-x}, written in increasing order on $d$ for $1\leq d\leq m-1$.
Then the Sato-Tate group of the Jacobian of the curve $C_{2^m}: y^2=x^{2^m}-1$ satisfies
$$ \ST(\Jac(C_{2^m})) \simeq \left\langle \U(1)\times  \left(\U(1)^{2^{m-2}-1}\right)_2, \gamma, \gamma_J \right\rangle$$ 
for $m\geq 3$. 
\end{restatable*}

There are some challenges in generalizing this result to twists of the curves, even when we have explicit models for the twists and for the maps between the curves. For example, we can easily write down a map from the curve $y^2=x^{2^m}-1$ to the curve $y^2=x^{2^m}-c$, where $c\in \mathbb Q^*$, and obtain some information about the endomorphism field of the Jacobian of this twist. However, it is difficult to determine the exact field in general and it may be smaller than expected due to some subtleties with the field we twist by. For specific curves, one can use Magma \cite{Magma} to gain more information about the endomorphisms but this would not solve the problem in general. One could instead approach this problem by computing the twisting Sato-Tate groups of the curves, as was done in \cite{Arora2016,FiteLorenzoSutherland2018, FiteSuthlerland2014}, but this approach is more feasible for fixed genus than for the general case.

We demonstrate the result in Theorem \ref{thm:STx^2m-1}  with the following examples.

\begin{example}\label{example:x16}
We consider the genus $g=7$ curve $C_{16}:y^2=x^{16}-1$. In this case $m=4$ and the isogeny in Equation \eqref{eq:JacobianDecomposition} in Section \ref{sec:decomposition} becomes
$$\Jac(C_{16})\sim \Jac(C'_3)  \times \Jac(C'_5) \times \Jac(C'_9),$$ 
where $C'_n$ denotes the curve $y^2=x^n-x$. The result in Theorem \ref{thm:nondegeneracyx^2m-1} proves that the associated Hodge group decomposes as 
$$\Hg(\Jac(C_{16}))=\Hg(\Jac(C'_3))  \times \Hg(\Jac(C'_5)) \times \Hg(\Jac(C'_9))$$ and Proposition \ref{prop:ST0x^2m-1} tells us that the identity component of the Sato-Tate group is 
\begin{align*}
    \ST^0(\Jac(C_{16})) &\simeq \simequal{\ST^0(\Jac(C'_3))}{\U(1)}\times \simequal{\ST^0(\Jac(C'_5))}{\U(1)_2} \times \simequal{\ST^0(\Jac(C'_9))}{\U(1)_2 \times \U(1)_2}\\
    & \simeq \U(1)\times \U(1)_2 \times \U(1)_2 \times \U(1)_2 \\
    &\simeq \U(1) \times (\U(1)_2)^3
    \end{align*}
  where $(\U(1)_2)^3$ is defined in Equation \eqref{notation}.  
The endomorphism field of $\Jac(C_{16})$ is $\Q(\zeta_{16})$; since the Jacobian is nondegenerate, this implies that the component group of the Sato-Tate group is isomorphic to $\Gal(\Q(\zeta_{16})/\Q)$.  We let  $\gamma_J=\diag(J,J,J,J,J,J,J)$ for $J:=\begin{pmatrix}0&1\\-1&0\end{pmatrix}$ and $ \gamma=\diag(\gamma_1, \gamma_2,\gamma_3)$, where
    $$\gamma_1=I, \hspace{.25in} \gamma_2=\begin{pmatrix}0&-J\\J&0 \end{pmatrix}, \hspace{.25in} 
 \gamma_3=\begin{pmatrix}0&0&I&0\\J&0&0&0\\0&0&0&-J\\0&I&0&0 \end{pmatrix}$$
    are the matrices defined in Theorem \ref{thm:STx^n-x} for the curves $C'_3,C'_5$, and $C'_9$ respectively.
    Hence, by Theorem \ref{thm:STx^2m-1}, $$\ST(\Jac(C_{16})) \simeq \left\langle\U(1)\times (\U(1)_2)^3,\, \gamma, \gamma_J \right\rangle$$ where $\gamma$ is 
    \[
\gamma= {\footnotesize\left(\begin{tikzpicture}[baseline=0cm]
    \matrix [mymatrix] (m)  
           {I\\
           {}&0&\text{-}J\\{}&J&0\\
           {}&{}&{}&0&0&I&0\\{}&{}&{}&J&0&0&0\\{}&{}&{}&0&0&0&\text{-}J\\{}&{}&{}&0&I&0&0\\};
\mymatrixbox{1}{1}{1}{1}
\mymatrixbox{2}{2}{3}{3}
\mymatrixbox{4}{4}{7}{7}
     \end{tikzpicture}\right)}.
\]

\end{example}

A nice feature of this family of curves is that as $m$ grows, the decomposition of   $\Jac(C_{2^m})$ simply gains additional factors of the form $\Jac(C'_{2^d+1})$. This leads to recursive definitions of the generators of the Sato-Tate group. To see this, we provide the next example.

\begin{example}\label{example:x32}

We consider the genus 15 curve $C_{32}:y^2=x^{32}-1$. In this case $m=5$ and the isogeny in Equation \eqref{eq:JacobianDecomposition} becomes$$\Jac(C_{32})\sim \Jac(C'_3)  \times \Jac(C'_5) \times \Jac(C'_9) \times \Jac(C'_{17}).$$
Proposition \ref{prop:ST0x^2m-1} tells us that the identity component of the Sato-Tate group is 
\begin{align*}
    \ST^0(\Jac(C_{32})) &\sim \simequal{\ST^0(\Jac(C'_3))}{\U(1)}\times \simequal{\ST^0(\Jac(C'_5))}{\U(1)_2} \times \simequal{\ST^0(\Jac(C'_9))}{(\U(1) \times \U(1))_2}\times \simequal{\ST^0(\Jac(C'_{17}))}{(\U(1) \times \U(1) \times \U(1) \times \U(1))_2}\\
    & \simeq \U(1)\times (\U(1)_2)^7.
    \end{align*}
We let $\gamma_J=\diag(J,J,\ldots,J)$ for  $J:=\begin{pmatrix}0&1\\-1&0\end{pmatrix}$ and $\gamma=\diag(\gamma_1, \gamma_2,\gamma_3, \gamma_4)$ whose entries are the matrices defined in Theorem \ref{thm:STx^n-x} for the curves $C'_3,\,C'_5,\,C'_9$, and $C'_{17}$ respectively.
Hence, by Theorem \ref{thm:STx^2m-1}, $$\ST(\Jac(C_{32})) \simeq  \left\langle \U(1)\times (\U(1)_2)^7,\, \gamma,\gamma_J \right\rangle,$$
where $\gamma$ is 
\[
\gamma= {\footnotesize\left(\begin{tikzpicture}[baseline=0cm]
    \matrix [mymatrix] (m)  
           {I\\
           {}&0&\text{-}J\\{}&J&0\\
           {}&{}&{}&0&0&I&0\\{}&{}&{}&J&0&0&0\\{}&{}&{}&0&0&0&\text{-}J\\{}&{}&{}&0&I&0&0\\
           {}&{}&{}&{}&{}&{}&{}&0&0&I&0&0&0&0&0\\
           {}&{}&{}&{}&{}&{}&{}&0&0&0&0&0&0&0&I\\
           {}&{}&{}&{}&{}&{}&{}&0&0&0&J&0&0&0&0\\
           {}&{}&{}&{}&{}&{}&{}&0&I&0&0&0&0&0&0\\
           {}&{}&{}&{}&{}&{}&{}&0&0&0&0&0&0&I&0\\
           {}&{}&{}&{}&{}&{}&{}&0&0&0&0&\text{-}J&0&0&0\\
           {}&{}&{}&{}&{}&{}&{}&I&0&0&0&0&0&0&0\\
           {}&{}&{}&{}&{}&{}&{}&0&0&0&0&0&I&0&0\\};
\mymatrixbox{1}{1}{1}{1}
\mymatrixbox{2}{2}{3}{3}
\mymatrixbox{4}{4}{7}{7}
\mymatrixbox{8}{8}{15}{15}

     \end{tikzpicture}\right)}.
\]

\end{example}

\subsection*{Organization of the Paper}

In Section \ref{sec:background} we review and establish necessary facts about Sato-Tate groups and nondegeneracy. In Section \ref{sec:nondegeneracy}, we prove the nondegeneracy of the Jacobians of $C_{2^m}$ and $C'_{2^d+1}$. We achieve this for $\Jac(C_{2^m})$ in part by using the decomposition of the Jacobian given in  Section \ref{sec:decomposition} and by analyzing the Hodge groups of the factors $\Jac(C'_{2^d+1})$. This enables us to use the techniques of \cite{EmoryGoodson2022, GoodsonCatalan} to determine generators of the Sato-Tate groups of the Jacobians of the curves $C_{2^m}$ and $C'_{2^d+1}$ in Section \ref{sec:STgroups}.  We give explicit examples of some of these generators and the Sato-Tate groups in Example \ref{example:x16} and Example \ref{example:x32} above. In Section \ref{sec:moments}, we compute moment statistics associated to the Sato-Tate groups of the Jacobians of $C_{2^m}$ and $C'_{2^d+1}$. These moment statistics can be used to verify the equidistribution statement of the generalized Sato-Tate conjecture by comparing them to the moment statistics obtained for the traces $a_i$ in the normalized $L$-polynomial. Note that the numerical moment statistics are an approximation since one can only ever compute them up to some prime. It is of interest to those dealing with equidistribution statements to compare how close these two computations are, and we do this for some examples of our curves in Section \ref{sec:moments}.

\subsection*{Notation and conventions} \label{notation} We briefly summarize the basic notation and terminology used in this paper.  Let $C$ be a smooth projective curve defined over $\mathbb Q$. We write $\End(\Jac(C)_F)$ for the ring of endomorphisms of the Jacobian of  $C$ that are  defined over the field $F$. Let $K:=K_C$ denote the minimal extension of $\mathbb Q$ over which all the endomorphisms of the abelian variety $\Jac(C)$ are defined, i.e. the minimal extension $L/\Q$ for which $\End(\Jac(C)_L)\simeq \End(\Jac(C)_{\overline{\mathbb Q}})$; the field $K$ is called the \textbf{endomorphism field} of $\Jac(C)$. When dealing with primitive roots of unity, which we denote by $\zeta_d$ for the $d$-th roots of unity, we mean $\zeta_d=e^{2\pi i/d}.$

We denote the Sato-Tate group of the Jacobian of $C$  by $\ST(\Jac(C)):=\ST(\Jac(C)_{\mathbb Q})$ with identity component denoted $\ST^0(\Jac(C))$ and component group  $\ST(\Jac(C))/\ST^0(\Jac(C))$.  For any rational number $x$ whose denominator is coprime to an integer $r$, $\langle x \rangle_r$ denotes the unique representative of $x$ modulo $r$ between 0 and $r-1$. 

Define the two matrices
$$I:=\begin{pmatrix}1&0\\0&1\end{pmatrix}\; \text{ and }\; J:=\begin{pmatrix}0&1\\-1&0\end{pmatrix}.$$ The symplectic form considered throughout the paper is given by $\diag(J,\dots,J).$ We embed $\U(1)$ in $\SU(2)$ via $u\mapsto U= \diag(u, \overline u)$. Lastly, for any positive integer $n$, we define the following subgroups of the unitary symplectic group $\USp(2n)$. 
\begin{equation*}\label{eqn:U1subn}
    \U(1)_n:=\left\langle \diag(\underbrace{u, \overline u, \ldots, u,\overline u}_{n-\text{times}}): u\in \mathbb C^\times, |u|=1\right\rangle,
\end{equation*}

\begin{equation*}\label{eqn:U1n}
    \U(1)^n:=\left\langle \diag( u_1,\overline{u_1},\ldots, u_n,\overline{u_n}):u_i\in \mathbb C^\times, |u_i|=1\right\rangle,
\end{equation*}
and

\begin{equation}\label{eqn:U1nk}
    (\U(1)_2)^r:=\underbrace{\U(1)_2 \times \U(1)_2 \times \cdots \times \U(1)_2}_{r-\text{times}}, 
\end{equation}
where $2r=n.$

Let $F$ be a CM field, and let $T_F$ denote the Weil restriction of scalars $R_{F/\mathbb Q}(\mathbb G_m)$. Then $T_F$ is an $[F:\mathbb Q]$-dimensional algebraic torus, and we let $U_F$ be the algebraic subtorus defined by the condition
\begin{equation}\label{eqn:UFtorus}
    U_F(\mathbb Q):=\{a\in T_F(\mathbb Q)\,|\, x\overline x=1\}.
\end{equation}

\subsection*{Acknowledgements} The authors would like to thank the anonymous reviewer for their helpful comments on earlier drafts of this paper. The first named author was supported by NSF grant DMS-2002085 and an AMS-Simons Travel Award. The second named author was supported by NSF grant DMS-2201085 and by a PSC-CUNY Award, jointly funded by The Professional Staff Congress and The City University of New York. This project began while the second author was a visiting research scientist at the Max-Planck-Institut f\"ur Mathematik. She is grateful for the supportive research environment that they provided.


\section{Background and Preliminaries}\label{sec:background}

\subsection{An $\ell$-adic Construction of the Sato-Tate Group}

We follow the exposition of \cite{EmoryGoodson2022,GoodsonDegeneracy} and \cite[Section 3.2]{SutherlandNotes}. See also \cite[Chapter 8]{SerreNXP}.

Let $A/F$ be a $g$-dimensional abelian variety defined over a number field $F$ with a fixed embedding $F\hookrightarrow \mathbb C$. For any prime $\ell$, we define the Tate module $T_{\ell}:=\varprojlim_{n} A[\ell^n]$; this is a free $\mathbb{Z}_{\ell}$-module of rank $2g$. We define the rational Tate module $V_{\ell}:=T_{\ell}\otimes_{\mathbb{Z}} \mathbb{Q}$, which is a $\mathbb{Q}_{\ell}$-vector space of dimension $2g.$ The Galois action on the Tate module is given by an $\ell$-adic representation 
\begin{align}\label{eqn:artinrepresntation}
    \rho_{A,\ell}:\Gal(\overline{F}/F) \rightarrow \Aut(V_{\ell}) \cong \GL_{2g,\mathbb{Q}_{\ell}}(\mathbb{Q}_{\ell}).
\end{align}
The $\ell$-adic monodromy group of $A$, denoted $G_{A,\ell}$, is the Zariski closure of the image of this map in $\GL_{2g,\mathbb{Q}_{\ell}}$, and we define $G^{1}_{A,\ell}:=G_{A,\ell}\cap \Sp_{2g,\mathbb{Q}_{\ell}}$. 

\begin{conjecture}[Algebraic Sato-Tate Conjecture]\cite[Conjecture 2.1]{Banaszak2015}\label{conjec:AST}  There is an algebraic subgroup $\AST(A)$ of $\GSp_{2g}$ over $\mathbb Q$, called the {algebraic Sato-Tate group of $A$}, such that the connected component of the identity $\AST^0(A)$ is reductive and, for each prime $\ell$, $G^{1}_{A,\ell}=\AST(A)\otimes_{\mathbb Q} \mathbb Q_{\ell}$.
\end{conjecture}

When this conjecture holds, we define the {Sato-Tate group} of $A$, denoted   $\ST(A)$, to be a maximal compact Lie subgroup of $\AST(A)(\mathbb{C})$ contained in $\USp(2g)(\mathbb{C})$. It is conjectured that $\ST(A)$ is, up to conjugacy in $\USp(2g)(\mathbb{C})$, independent of the choice of the prime $\ell$ and of the embedding of $\mathbb{Q}_\ell$ in $\mathbb{C}$. While the Sato-Tate group is a compact Lie group, it may not be connected \cite{FiteSutherland2016}. We denote  the connected component of the  identity  (also called the identity component) of $\ST(A)$ by  $\ST^0(A)$.

There are many cases where the algebraic Sato-Tate conjecture is known to be true.  Banaszak and Kedlaya prove that the conjecture holds for all abelian varieties of dimension at most 3 \cite[Theorem 6.11]{Banaszak2015} and for many examples of simple abelian varieties  \cite[Theorem 6.9]{Banaszak2015}. There are examples of infinite families of higher dimensional Jacobian varieties for which the algebraic Sato-Tate conjecture is known to be true in \cite{EmoryGoodson2022, FiteSutherland2016, GoodsonCatalan}. Furthermore, Cantoral-Farf\'an and Commelin \cite{CantoralCommelin2022} proved that the algebraic Sato-Tate conjecture holds whenever the Mumford-Tate conjecture holds for an abelian variety.  Note that the abelian varieties that we consider in this paper are CM and the Mumford-Tate conjecture is known for CM abelian varieties and so, as a consequence, the algebraic Sato-Tate conjecture is known to hold for $\Jac(C_{2^m})$ and $\Jac(C'_{2^d+1})$.

\subsection{Hodge and Mumford Tate Groups}\label{sec:backgroundHodgeMT}

In this section, we follow the exposition in \cite{GoodsonDegeneracy} and Chapters 1 and 17 of \cite{BirkenhakeLange2004}. 

Let $A$ be a nonsingular projective variety over $\mathbb C$. We denote the first homology group of $A$ by $V(A):=H_1(A,\Q)$ and its dual (the first cohomology group) by $V^*(A):=H^1(A,\Q)$. The complex vector space $V(A)_\C$ has a weight $-1$ Hodge structure, i.e., a decomposition $V(A)_\C= V(A)^{-1,0}\oplus V(A)^{0,-1}$ where  $\overline{V(A)^{-1,0}}=V(A)^{0,-1}$. This corresponds to the following weight 1 Hodge structure of its dual
$$V^*_\C =H(A)^{1,0}\oplus H(A)^{0,1},$$
where $V(A)^{-1,0}=H^{1,0}(A)^*$ and $V(A)^{0,-1}=H^{0,1}(A)^*$. The notation in the decomposition of $H^1(A,\Q)$ is defined by $H^{a,b}(A)=H^b(\Omega^a_A),$ where $\Omega_A^a$ is the sheaf of holomorphic $a$-forms on $A$. We can also define $H^{a,b}(A)$ by
\begin{align*}
    H^{a,b}(A)\simeq \bigwedge^a H^{1,0}(A) \otimes \bigwedge^b H^{0,1}(A).
\end{align*}
Hodge structures of weight $n$ are defined by $H^n(A,\C):=\bigoplus_{a+b=n} H^{a,b}(A).$

The Hodge structure  determines a representation $\mu_{\infty,A}: \mathbb G_{m}(\C) \rightarrow \GL(V_\C)$ acting as multiplication by $z$ on $V(A)^{-1,0}$ and trivially on $V(A)^{0,-1}$. 
With this setup, we define the {Mumford-Tate group} of $A$, denoted $\MT(A)$, to be the smallest $\Q$-algebraic subgroup of $\GL_{V}$  such that $\mu_{\infty,A}(\mathbb G_{m}(\C))\subseteq \MT(A)(\C)$. We define the Hodge group of $A$, denoted $\Hg(A)$, to be the connected component of the identity of $\MT(A)\cap \SL_{V}$. 

The Hodge group can also be formed by restricting the representation $\mu_{\infty,A}$ to the circle group $\mathbb S^1:=\{z\in\C\mid |z|=1\}$. With this setup, the Hodge group is the smallest $\Q$-algebraic subgroup of $\GL_{V}$ such that  $\mu_{\infty,A}(\mathbb S^1)\subseteq \Hg(A)_\C$. The image of this restriction of $\mu_{\infty,A}$ lies in $\SL(V_\C)$, and so the Hodge group is a $\Q$-algebraic subgroup of $\SL_{V}$. In fact, one can show that the image of the representation $\mu_{\infty,A}$ is contained in the symplectic group $\GSp({V_\C})$. Hence, $\MT(A)$ and $\Hg(A)$ are $\Q$-algebraic subgroups of $\GSp_{V}$ and $\Sp_{V}$, respectively.

There are some useful identities for Hodge groups of products of abelian varieties. For $n\geq 1$, we can identify $\Hg(A^n)$ with $\Hg(A)$ and the action is performed diagonally on $V(A^n)=(V(A))^n$. More generally, for $n_1,n_2,\ldots, n_t\geq 1$, $\Hg(A_1^{n_1}\times A_2^{n_2}\times \cdots \times A_t^{n_t})   $ is isomorphic to $\Hg(A_1\times A_2\times \cdots \times A_t)$. Even more generally, if $A$ and $B$ are abelian varieties then $\Hg(A\times B)\subseteq \Hg(A)\times \Hg(B)$ (see \cite{GoodsonDegeneracy} for examples where this containment is strict).

\subsection{Nondegenerate Abelian Varieties}\label{sec:NondegenerateAV}

We denote the (complexified) Hodge ring of $A$ by
$$\mathscr B^*(A):=\displaystyle\bigoplus_{d=0}^{\dim(A)} \mathscr B^d(A), $$
where $\mathscr B^d(A)=(H^{2d}(A,\mathbb Q)\cap H^{d,d}(A))\otimes \mathbb C$ is the $\mathbb C$-span of Hodge cycles of codimension $d$ on $A$.  
The subring of $\mathscr B^*(A)$ generated by the divisor classes, i.e.  generated by $\mathscr B^1(A)$ , is
$$\mathscr D^*(A):=\displaystyle\bigoplus_{d=0}^{\dim(A)} \mathscr D^d(A),$$ 
where $\mathscr D^d(A)$ is the $\mathbb C$-span of classes of intersection of $d$ divisors. Note that we always have the containment $\mathscr D^*(A) \subseteq \mathscr B^*(A)$, and we say that an abelian variety $A$ is \emph{nondegenerate} if the containment is an equality, i.e. its Hodge ring is generated by divisor classes.

Results of Hazama \cite{Hazama85} show that $A$ is stably nondegenerate if and only if the rank of the Hodge group is maximal, which implies a similar statement for the Mumford-Tate group  ($A$ is stably nondegenerate if $\mathscr D^*(A^n) = \mathscr B^*(A^n)$ for all $n>0$). Hazama also proves an interesting result regarding products of abelian varieties: if $A$ and $B$ are both stably nondegenerate, then the only scenario in which $A\times B$ could be degenerate is if at least one of the simple factors is of type IV in the Albert's classification.

The relationship between divisor classes and Hodge cycles is relevant to the Hodge Conjecture. Let $\mathscr C^d(A)$ be the subspace of $\mathscr B^d(A)$ generated by the classes of algebraic cycles on $A$ of codimension $d$. Then 
$$\mathscr D^d(A) \subseteq \mathscr C^d(A) \subseteq \mathscr B^d(A)$$
and the Hodge Conjecture for $A$ asserts that every Hodge cycle is algebraic: $\mathscr C^d(A) = \mathscr B^d(A)$ for all $d$ (see \cite{Aoki2002, Shioda82}). One way to prove the Hodge Conjecture in codimension $d$ is to prove the equality $\mathscr D^d(A) =\mathscr B^d(A)$. However this equality does not always hold, even when the Hodge conjecture is known to be true. 

An example of this phenomenon is given by Shioda \cite{Shioda82} (and worked out in greater detail in \cite{EmoryGoodson2022, GoodsonDegeneracy,LombardoMT2021}). The Jacobian of the curve $y^2=x^9-1$ satisfies the Hodge Conjecture and is a 4-dimensional abelian variety that is isogenous to the product of two nondegenerate abelian varieties: a CM elliptic curve $E$ and a 3-dimensional absolutely simple CM abelian variety $A$. However, Shioda shows that the Hodge ring in dimension 2 contains exceptional cycles (those not generated by divisor classes). Furthermore, Lombardo shows that the Mumford-Tate group of the Jacobian is degenerate and the degeneracy of the Sato-Tate group is examined in \cite{EmoryGoodson2022,GoodsonDegeneracy}. See \cite{GoodsonDegeneracy} for more examples of this.

Nondegeneracy also plays a role in the component group of the Sato-Tate group. In general, for an abelian variety $A/F$ we have a canonical surjection
$$\ST(A)/\ST^0(A) \rightarrow \Gal(K/F),$$
where $K$ is the endomorphism field of $A$. If $A/F$ is a stably nondegenerate abelian variety then this surjection is an isomorphism (see, for example, \cite[Remark 6.4]{Banaszak2015}). On the other hand, the surjection is not necessarily an isomorphism if $A$ is degenerate.  In general, we have an isomorphism $\ST(A)/\ST^0(A) \simeq \Gal(L/F)$, where $L$ is the minimal Galois extension of $F$ for which $\ST(A_L)$ is connected. The field $L$ is given by the kernel of the representation $\Gal(\overline{F}/F)\rightarrow G_{A,\ell} \rightarrow G_{A_\ell}/G^0_{A,\ell}$ which is induced by the $\ell$-adic representation $\rho_{A,\ell}$ in Equation \eqref{eqn:artinrepresntation} (see \cite[Theorem 3.12]{{SutherlandNotes}}). See \cite[Section 5.2]{GoodsonDegeneracy} for an example of a Jacobian variety  for which the field $L$ is  larger than its endomorphism field.

Furthermore, for stably nondegenerate abelian varieties, we can determine explicit generators of the component group of the Sato-Tate group through the twisted Lefschetz group. The twisted Lefschetz  group, denoted $\TL(A)$, is a closed algebraic subgroup of  $\Sp_{2g}$ defined by
\begin{align*}
\TL(A):=\bigcup_{\tau \in \Gal(\overline{F}/F)} \LL(A)(\tau),
\end{align*}
where $\LL(A)(\tau):=\{\gamma \in \Sp_{2g}\mid \gamma \alpha \gamma^{-1}=\tau(\alpha) \text{ for all }\alpha \in \End(A_{\overline{F}})_\mathbb{Q}\}$.

The following proposition holds by applying the proof techniques used in \cite[Lemma 3.5]{FGL2016}, which relies on results from \cite{Banaszak2015} and \cite{Fite2012}.
\begin{proposition}\label{prop:AST=TL}
The algebraic Sato-Tate Conjecture holds for stably nondegenerate abelian varieties $A$ with $\AST(A)=\TL(A)$.
\end{proposition}

We rely on this result in Section \ref{sec:STgroups} when we compute generators of the Sato-Tate groups of the two families of Jacobian varieties we focus on in this paper.


\section{Nondegeneracy}\label{sec:nondegeneracy}
In this section we prove that the Jacobian varieties we are studying are nondegenerate. We do this by first determining the decomposition of the Jacobians in Section \ref{sec:decomposition} and then by analyzing the Hodge groups of the factors in the decomposition in Section \ref{sec:nondegeneracyproof}. 
 The results in this section allow us to use the twisted Lefschetz group to determine generators of the Sato-Tate group for each Jacobian in Section \ref{sec:STgroups}.

\subsection{Decompositions of the Jacobians}\label{sec:decomposition}
Here we specify known results for decompositions of Jacobians to the families we are studying in this paper. We show that the Jacobians of curves of the form $C_{2^m}: y^2=x^{2^m}-1$ decompose into products of Jacobians of curves of the form $C_{2^d+1}':y^2=x^{2^d+1}-x$, which gives the motivation for studying these two specific families of curves together.  

We first recall the following result, which we will use to describe the decomposition of the Jacobians of the curves $C_{2^m}$.
\begin{theorem}\cite[Theorem 4.3]{EmoryGoodsonPeyrot}\label{thm:JacobianDecomposition}
Let $c\in\mathbb Q^*$, and let $v_2\colon  \mathbb Q^* \rightarrow \mathbb Z$ denote the $2$-adic valuation map. Let $C\colon y^2= x^{2g+2}+c$ be a hyperelliptic curve of genus $g$ and write $k:= v_2(g+1)$. Then we have the following isogeny over $\Q$
$$\Jac(C) \sim \Jac(y^2= x^{{(g+1)}/{2^k}}+c)^2 \times \prod_{i=0}^{k-1} \Jac(y^2=x^{{(g+1)}/{2^{i}}+1}+cx).$$
\end{theorem}

The genus of the curve $y^2=x^{2^m}-1$ is $g=2^{m-1}-1$, and so applying  Theorem \ref{thm:JacobianDecomposition} to this curve and adjusting the indices yields 
\begin{align}\label{eq:JacobianDecomposition}
    \Jac(y^2=x^{2^m}-1) \sim \displaystyle\prod_{d=1}^{m-1}\Jac(y^2=x^{2^{d}+1}-x).
\end{align}
The Jacobians of dimension greater than 1 (i.e., when $d>1$) that appear in this decomposition can be factored further.

\begin{proposition}\label{prop:FullJacDecomp}Let $m \geq 3$ and $1 \leq d \leq m-1$. 
    The Jacobian of $C_{2^m}:y^2=x^{2^m}-1$ factors over $\overline\Q$ as the product 
    $$\Jac(y^2=x^{2^m}-1)\sim (y^2=x^3-x)\times\prod_{d=2}^{m-1} X_d^2,$$
    where each $X_d$ is a simple Jacobian variety of dimension $2^{d-2}$ with CM by $\Q(\zeta_{2^{d+1}}-\overline\zeta_{2^{d+1}})$ and $\zeta_{2^{d+1}}$ is a primitive $2^{d+1}$th root of unity.
\end{proposition}

\begin{proof}
    From Equation \eqref{eq:JacobianDecomposition} with $c=1$ we have the following isogeny over $\mathbb{Q}$  $$\Jac(y^2=x^{2^m}-1)\sim (y^2=x^3-x)\times\displaystyle\prod_{d=2}^{m-1}\Jac(y^2=x^{2^{d}+1}-x).$$  
    Let $\zeta$ be a primitive $2^{d+1}$th root of unity. It is shown in Proposition 3.10, and Corollary 4.6 of \cite{Carocca2011} that the $\Jac(y^2=x^{2^{d}+1}-x)$ is isogenous to the square of a simple abelian variety with CM by $\Q(\zeta+\zeta^{2^d-1})$ . Combining this with the identity $\zeta^{2^d-1}=-\overline\zeta$ yields the result. 
   \end{proof}
 
\begin{remark}
    Just as in \cite[Remark 5.4]{FiteSutherland2016}, we note that the Sato-Tate group of $\Jac(C_{2^m})$ is not isomorphic to the direct sum of the Sato-Tate groups of its factors. Indeed, it cannot be the case since as we show in Section \ref{sec:nondegeneracyproof} these are all nondegenerate abelian varieties and so, by the discussion in Section \ref{sec:NondegenerateAV}, the component groups of the Sato-Tate groups are isomorphic to the Galois groups of the endomorphism field(s) over $\Q$. However, the Galois group associated to $\Jac(C_{2^m})$ is not isomorphic to the direct product of the Galois groups of the factors. This further demonstrates the need to give an
explicit description for $\ST(C_{2^m})$ and $\ST(C'_{2^d+1})$ in terms of generators. 
\end{remark}

\subsection{Nondegeneracy Results}\label{sec:nondegeneracyproof}

We now prove that the Jacobians of the curves $ y^2=x^{2^m} -1$ and $y^2=x^{2^{d}+1}-x$ and their twists are nondegenerate. 
\begin{theorem}\label{theorem:nondegeneracyx^n-x}
Let $d\geq1$. The Jacobian of $C'_{2^d+1}:y^2=x^{2^d+1}-x$ is stably nondegenerate.    
\end{theorem}
\begin{proof}
When $d=1$, the curve $C'_{2^d+1}$ is an elliptic curve and it is well-known that elliptic curves are nondegenerate. Thus, for the remainder of the proof, we assume that $d>1$. 

It is shown in the proof of Proposition \ref{prop:FullJacDecomp} that the Jacobian of $C'_{2^d+1}$ is isogenous to the square of a simple abelian variety, which we denote by $X_d$. Corollary 4.6 of \cite{Carocca2011} shows that the CM-field of $X_d$ is $F_d:=\Q(\zeta+\zeta^k)$, where $\zeta:=\zeta_{2^{d+1}}$ is a primitive $2^{d+1}$th root of unity and $k=2^d-1$. It is further shown in Proposition 4.7 of \cite{Carocca2011} that this field does not contain a proper CM-subfield. Thus, by Theorem 1.1 of \cite{Kida2020}, every CM-type of the field $F_d$ is nondegenerate. In particular, the CM-type of $X_d$ is nondegenerate. Results of Hazama \cite{Hazama83} prove that this is equivalent to the abelian variety $X_d$ being stably nondegenerate (see also, \cite[Proposition 2.2]{Aoki2004}). Thus, since $\Jac(C'_{2^d+1})$ is isogenous to $X_d^2$, we have that $\Jac(C'_{2^d+1})$ is stably nondegenerate.
\end{proof}

\begin{corollary}
   Let $X_d$ be the abelian variety defined in Proposition \ref{prop:FullJacDecomp}. Then $\Hg(X_d)= U_{F_d}$.
\end{corollary}

\begin{proof}\label{cor:HodgeUFd}
Each $X_d$ is a CM abelian variety, and so its Hodge group is a commutative algebraic torus. We can, therefore, write $\Hg(X_d)\subseteq U_{F_d}$, where $U_{F_d}$ is as defined in Equation \eqref{eqn:UFtorus}. Furthermore, since $X_d$ is nondegenerate, the dimension of the Hodge group is maximal i.e., $\dim(\Hg(X_d))=\dim(X_d)$ since $X_d$ is simple. Lastly, $\dim(X_d)=1/2 [F_d:\Q]$, which equals $\dim(U_{F_d})$ (see, for example, \cite[Section 7]{MoonenZarhin95}). This proves the desired result.
\end{proof}

\begin{theorem}\label{thm:nondegeneracyx^2m-1}
    Let $m\geq 3$. The Jacobian of $y^2=x^{2^m}-1$ is stably nondegenerate.
\end{theorem}

\begin{proof}
We will prove that the Hodge group of $\Jac(y^2=x^{2^m}-1)$ is maximal (i.e., the rank of $\Hg(\Jac(y^2=x^{2^m}-1))$ equals the reduced dimension of $\Jac(y^2=x^{2^m}-1)$, which is equivalent to proving that its Hodge ring is generated by divisor classes (see \cite[Definition 2.6 and Theorem 2.7]{Hazama85}). The decomposition in Proposition \ref{prop:FullJacDecomp} implies that $\Hg(\Jac(y^2=x^{2^m}-1))=\Hg(X_1\times X_2^2 \times \cdots \times X_{m-1}^2)$, where $X_1$ denotes the elliptic curve  $y^2=x^3-x$. Recall from Section \ref{sec:backgroundHodgeMT} that, for an abelian variety $A$, we have $\Hg(A^n)=\Hg(A)$ for $n\geq 1$. Hence, to prove the stable nondegeneracy of $\Jac(y^2=x^{2^m}-1)$ it suffices to prove that 
    $$\Hg(X_1\times X_2^2 \times \cdots \times X_{m-1}^2)= \Hg(X_1)\times\Hg(X_2)\times\cdots\times\Hg(X_{m-1}).$$
    
The elliptic curve $X_1$ has complex multiplication by $F_1:=\Q(\zeta_4)=\Q(i)$. For $d\geq 2$, the simple factor $X_d$ has complex multiplication by the field $F_d=\Q(\zeta+\zeta^k)$,  where $\zeta:=\zeta_{2^{d+1}}$ is a primitive $2^{d+1}$th root of unity and $k=2^d-1$. As noted in the proof of Theorem \ref{theorem:nondegeneracyx^n-x}, the field $F_d$, for $d\geq 1$, does not contain a proper CM-subfield. Hence, for distinct simple factors $X_{d_1}$ and $X_{d_2}$ with $d_1<d_2$, there is no embedding $F_{d_1}\hookrightarrow F_{d_2}$.

Now let $X=X_{d_2}\times \cdots \times X_{d_r}$ be a subproduct of $X_1\times \cdots \times X_{m-1}$. For any $d_1<d_2,d_3,\ldots,d_r$, the abelian variety $X_{d_1}$ is not a factor of $X$ and its dimension is less than or equal to the dimension of each factor of $X$. We will prove that 
        $$\Hg(X_{d_1}\times X)=\Hg(X_{d_1})\times \Hg(X).$$

Since the simple abelian varieties  $X_{d_2}, \ldots, X_{d_r}$ are nonisogenous, the endomorphism algebra of their product $X$ is the product $F_{d_2} \times \ldots \times F_{d_r}$ and $\Hg(X)$ is contained in $U_{F_{d_2}}\times\cdots\times U_{F_{d_r}}$.

As noted in Lemma 3.6 of \cite{MoonenZarhin99}, if $\Hg(X_{d_1}\times X)$ is strictly contained in $\Hg(X_{d_1})\times \Hg(X)$ then the center of $\Hg(X)$ contains an algebraic torus that is $\Q$-isogenous to $\Hg(X_{d_1})$. This would imply that there is a homomorphism 
    $$U_{F_{d_1}} \to U_{F_{d_2}}\times\cdots\times U_{F_{d_r}}$$
with finite kernel. However this is not possible since $F_{d_1}$ does not embed into any of the fields $F_{d_i}$. Thus, we must have 
$$\Hg(X_{d_1}\times X)=\Hg(X_{d_1})\times \Hg(X).$$

Given that the above holds for any subproduct $X$, we can conclude that 
    $$\Hg(X_1\times \cdots \times X_{m-1})=\Hg(X_1)\times\cdots\times\Hg(X_{m-1}).$$
    Since each $X_{d_i}$ is nondegenerate (see Theorem \ref{theorem:nondegeneracyx^n-x}), the Hodge group of each has maximal dimension. Thus, the Hodge group of $\Jac(y^2=x^{2^m}-1)$ is also maximal, and so by Theorem 2.7 of \cite{Hazama85} the Jacobian of $y^2=x^{2^m}-1$ is stably nondegenerate.
\end{proof}

\begin{corollary}\label{cor:nondegeneracyboth}
The Jacobian of any twist of  $C_{2^m}$ or $C_{2^d+1}$ is stably nondegenerate.
\end{corollary}
\begin{proof}
This follows from Theorems \ref{theorem:nondegeneracyx^n-x} and \ref{thm:nondegeneracyx^2m-1} since nondegeneracy is invariant under twisting.
\end{proof}

\begin{remark}
    In particular, the Jacobians of $C_{2^m}:y^2=x^{2^m}-c$ and $C_{2^d+1}:y^2=x^{2^d+1}-cx$ are stably  nondegenerate for any $c\in\Q^\times$. 
\end{remark}


\section{Sato-Tate Groups}\label{sec:STgroups}

The main results of this section are Theorems \ref{thm:STx^n-x} and \ref{thm:STx^2m-1} where we compute generators of the Sato-Tate groups of the Jacobians of the curves $C'_{2^d+1}$ and $C_{2^m}$. For both families of curves, we obtain the component group of the Sato-Tate group by computing the twisted Lefschetz groups (recall the results of Proposition \ref{prop:AST=TL}).

We begin by again considering the decomposition of the Jacobian of $y^2=x^{2^m}-1$, now with the goal of describing its complex uniformization.

\subsection{Explicit Maps to Lower Dimensional Jacobians}\label{sec:MapToLowerDim}

Theorem \ref{thm:JacobianDecomposition} and Equation \eqref{eq:JacobianDecomposition} of Section \ref{sec:nondegeneracy} give us a decomposition of the Jacobian of the curve $y^2=x^{2^m}-1$. In this subsection we give the maps to the lower dimensional Jacobian varieties since we will need these for our work in Section \ref{sec:Preliminariesx^2m-1}.

Let $C$ denote the curve $y^2=x^{2g+2}-1$. We can write the following nonconstant morphisms defined over $\Q$ to the lower genus curves $C':y^2=x^{g+2}-x$ and $\tilde{C}: y^2=x^{g+1}-1$:
\begin{align*}
    \phi&:C\to C', \;\;\;(x,y)\mapsto (x^2,xy),\\
    \tilde\phi&:C\to \tilde C,\;\;\;\, (x,y)\mapsto (x^2,y).
\end{align*}

The curves $y^2=x^{2g+2}-1$ in this paper are quite special: we are specifically working with curves $y^2=x^{2^m}-1$, which have odd genus $g=2^{m-1}-1$. This implies that the lower genus curves $C'$ are of the form $y^2=x^{2^{m-1}+1}-x$ and $\tilde{C}$ are of the form $y^2=x^{2^{m-1}}-1$. We can use the same process as above to write maps from $y^2=x^{2^{m-1}}-1$ to lower genus curves. In general, for $2\leq d\leq m-1$, there are maps from $C_{2^m}: y^2=x^{2^m}-1$ to $C'_{2^d+1}: y^2=x^{2^d+1}-x$ and $C_{2^d}: y^2=x^{2^d}-1$ given by
\begin{align}
    \phi_d&:C_{2^m}\to C'_{2^d+1}, \;\;\;(x,y)\mapsto (x^{2^{m-d}},x^{2^{m-d-1}}y),\label{eq:maptox^n-x}\\
    \tilde\phi_d&:C_{2^m}\to \tilde C_{2^d},\;\;\;\;\;\; (x,y)\mapsto (x^{2^{m-d}},y).\label{eq:maptox^2k-1}
\end{align}
There is an additional map $\phi_d$ when $d=1$. This maps $C_{2^m}$ to the elliptic curve $y^2=x^3-x$, which is isogenous to the curve $y^2=x^4-1$ obtained from the map $\tilde\phi_2$. For the factorization in Equation \eqref{eq:JacobianDecomposition}, we chose the model $y^2=x^3-x$ for the elliptic curve factor and so the maps we will work with are all of the form appearing in Equation \eqref{eq:maptox^n-x}. 

For $1\leq d\leq m-1$, the genus of the curve $C'_{2^d+1}: y^2=x^{2^d+1}-x$ is $g_d=2^{d-1}$. A basis for the space of regular differential 1-forms of $\Jac(C'_{2^d+1})$ is given by $\{x^j dx/y\}_{j=0}^{g_d-1}$. 
Pulling back these 1-forms via the map $\phi_d:C_{2^m}\to C'_{2^d+1}$ yields $\phi_d^*(x^j dx/y)=2^{m-d}x^{n_{d,j}} dx/y$, where $n_{d,j}=2^{m-d-1}(2j+1)-1$. Letting $\omega_{d,j}=x^{n_{d,j}} dx/y$, we see that the map $\phi_d$ corresponds to the inclusion
\begin{align}\label{eqn:UniformizationInclusion}
\C\left\langle\{\omega_{d,j}\}_{j=0}^{g_d-1}\right\rangle\hookrightarrow\C\left\langle\{x^{i} dx/y\}_{i=0}^{g-1}\right\rangle,    
\end{align}
where $\C\left\langle\{x^{i} dx/y\}_{i=0}^{g-1}\right\rangle=H^0((C_{2^m})_\C,\Omega^1_\C)$  can be identified with the complex uniformization of $\Jac(C_{2^m})$.

It is straightforward to verify that distinct values of $n_{d,j}$ lead to  distinct subspaces $\C\left\langle\{x^{n_{d,j}} dx/y\}\right\rangle$ and that the direct sum of these subspaces span $H^0((C_{2^m})_\C,\Omega^1_\C)$. Note that when $d=m-1$, the exponents $n_{d,j}$ are all even and the inclusion in Equation \eqref{eqn:UniformizationInclusion} yields all of the differentials in $H^0((C_{2^m})_\C,\Omega^1_\C)$ with even exponent for $x$.

\subsection{Sato-Tate groups for curves of the form $y^2=x^{2^d+1}-x$}

Let $C'_{2^d+1}$ denote the curve $y^2=x^{2^d+1}-x$, where $d\geq 2,$ and so the genus of $C'_{2^d+1}$ is a power of 2: $g_d=2^{d-1}$. 

\subsubsection{Preliminaries}\label{sec:Preliminariesx^n-x}
Throughout this subsection, let $\zeta_{2^{d+1}}$ be a primitive $2^{d+1}$th root of unity. It follows from work in \cite[Section 4.3]{Carocca2011} and \cite[Table 1]{Muller2022} that the group of automorphisms of $C'_{2^d+1}$ is generated by the automorphisms
\begin{align*}
    \alpha_d(x,y)&=(\zeta_{2^{d+1}}^2 x,\zeta_{2^{d+1}}y)&
    \beta_d(x,y)&=\left(-\frac{1}x,\frac{y}{x^{g_d+1}}\right).
\end{align*}
The endomorphism field of $\Jac(C'_{2^d+1})$ is $K=\Q(\zeta_{2^{d+1}})$  \cite[Proposition 4.2 and Theorem 2]{Carocca2011}.    

In Theorem  \ref{theorem:nondegeneracyx^n-x} and Corollary \ref{cor:nondegeneracyboth} we proved that the Jacobian variety $\Jac(C'_{2^d+1})$ is nondegenerate. This allows us to conclude that 
$\ST(\Jac(C'_{2^d+1}))/\ST^0(\Jac(C'_{2^d+1}))\simeq \Gal(K/\Q)$. The following lemma gives a generating set for this Galois group.
\begin{lemma}\label{lemma:generatorGalois}
    The Galois group $\Gal(\Q(\zeta_{2^{d+1}})/\Q)$ equals $\langle \sigma_{-1}, \sigma_5\rangle$, where $\sigma_a\colon
    \zeta_{2^{d+1}}\mapsto  \zeta_{2^{d+1}}^a$.
\end{lemma}

\begin{proof}
   The proof of this result relies on simple abstract algebra facts that can be found in \cite{DummitFoote}. The Galois group $\Gal(\Q(\zeta_{2^{d+1}})/\Q)$ is isomorphic to $(\Z/2^{d+1}\Z)^\times$ via $a\mapsto \sigma_a$. The group $(\Z/2^{d+1}\Z)^\times$ is generated by two elements, one of order 2 and one of order $2^{d-1}$. We can choose 5 as the element of order $2^{d-1}$. Since no power of 5 is congruent to $-1 \pmod{2^{d+1}}$ for $d\geq 1$, we can choose $-1$ as our generator of order 2. Thus, we can choose $\sigma_{-1}, \sigma_5$ as generators of $\Gal(\Q(\zeta_{2^{d+1}})/\Q)$.
\end{proof}

The final preliminary detail we need before computing generators of the Sato-Tate group is a description of certain endomorphisms of $\Jac(C'_{2^d+1})$. We compute pullbacks of the regular 1-forms $\omega_j=x^j dx/y$, where $j=0,1,\ldots, g_d-1$, with respect to the curve automorphisms $\alpha_d$ and $\beta_d$ and find that
\begin{align}
    \alpha_d^*(\omega_j)&=\zeta_{2^{d+1}}^{2j+1}\omega_j\label{eqn:pullbackalphad}\\
    \beta_d^*(\omega_j)&=(-1)^j\omega_{g_d-1-j}.\label{eqn:pullbackbetad}
\end{align}

We take the symplectic basis of $H_1(\Jac(C'_{2^d+1})_\C,\C)$ corresponding to the basis $\{\omega_j\}_{j=0}^{g_d-1}$ of regular 1-forms (this basis is symplectic with respect to the skew-symmetric matrix $\diag(J,J,\ldots, J)$) as in \cite[Lemma 3.6]{FGL2016}. With this setup, the endomorphisms corresponding to the automorphisms $\alpha_d$ and $\beta_d$ are
\begin{align}
\alpha_d&=\diag(Z_{d},Z_{d}^3,Z_{d}^5\ldots,Z_{d}^{2g_d-1})\label{eqn:alphad}\\
    \beta_d &=\antidiag(I, -I, I, \ldots, -I)\label{eqn:betad},
\end{align}
where $Z_{d}=\diag(\zeta_{2^{d+1}},\overline\zeta_{2^{d+1}})$.

\subsubsection{Sato-Tate group results}

We begin with a result for the identity component of the Sato-Tate group, which confirms Conjecture 6.13 of \cite{EmoryGoodsonPeyrot} for this family of curves.

\begin{proposition}\label{prop:ST0x^n-x}  Let $d \geq 2$ and $C'_{2^d+1}: y^2=x^{2^d+1}-x$. Then 

\[\ST^0(\Jac(C'_{{2^d+1}}))\simeq\left(\U(1)^{2^{d-2}}\right)_2 .
    \]
When $d=1$, $\Jac(C'_{{2^d+1}})$ is an elliptic curve with CM, and so $\ST^0(\Jac(C'_{{2^d+1}}))\simeq \U(1) $.
\end{proposition}

\begin{proof}
The result for $d=1$ is well-known. When $d\geq 2$, each of the Jacobians $\Jac(C'_{2^d+1})$ are isogenous over $\overline\Q$ to squares of simple abelian varieties with CM and are  of dimension $2^{d-1}$ (see Proposition \ref{prop:FullJacDecomp}). Since $\Jac(C'_{2^d+1})$  is stably nondegenerate (Corollary \ref{cor:nondegeneracyboth}), the result follows.      
\end{proof}

\begin{theorem}\label{thm:STx^n-x}
Let $g_d$ be the genus of the curve $C'_{{2^d+1}}:y^2=x^{2^d+1}-x$, where $d \geq 1$, and let $\zeta:=\zeta_{2^{d+1}}$ be a primitive $2^{d+1}$th root of unity. Up to conjugation in $\USp(2g_d)$, the Sato-Tate group of the Jacobian of the curve  satisfies
$$\ST(\Jac(C'_{{2^d+1}}))\simeq \left\langle \left(\U(1)^{2^{d-2}}\right)_2, \gamma_J, \gamma_d \right\rangle$$ 
where $\gamma_J=\diag(J,J,\ldots,J)$ and the $2\times2$ block entries of $\gamma_d$  are given by
$$\gamma_d[i,j] = \begin{cases}
I & \text{if } 2j-1=\langle 5(2i-1)\rangle_{2^{d+1}},\\
J & \text{if } j\leq\frac{g_d}{2} \text{ and } 2j-1=2^{d+1}-\langle 5(2i-1)\rangle_{2^{d+1}},\\
-J & \text{if } j>\ \frac{g_d}{2} \text{ and }  2j-1=2^{d+1}-\langle 5(2i-1)\rangle_{2^{d+1}},\\
0 & \text{otherwise,}
\end{cases}$$
for $1\leq i,j\leq g_d$.
\end{theorem}

\begin{proof}
For $d=1$, it is well-known that $\ST(\Jac(C'_{{2^d+1}}))\simeq N(\U(1))$. From the above construction we obtain the matrices $\gamma_J=J$ and $\gamma_1=I$, and we see that $N(\U(1))\simeq \langle \U(1), J\rangle$.

For $d\geq2$ we use the techniques implemented in \cite{EmoryGoodson2022, GoodsonCatalan} to compute the twisted Lefschetz group of the Jacobian. Applying Proposition \ref{prop:AST=TL} then yields the desired result.

Recall from Section \ref{sec:Preliminariesx^n-x} that the endomorphism field of $\Jac(C'_{2^d+1})$ is $K=\Q(\zeta_{2^{d+1}})$. Lemma \ref{lemma:generatorGalois} gives the two generators of the Galois group, and we now consider the action of each generator on the endomorphisms $\alpha_d$ and $\beta_d$. 

First note that the action on $\beta_d$ is trivial for both $\sigma_{-1}$ and $\sigma_5$. The block entries of $\alpha_d$ are odd powers of the $2\times2$ matrix $Z_{d}=\diag(\zeta,\overline\zeta)$. The Galois element $\sigma_{-1}$ acts on blocks of this form as follows:
$${}^{\sigma_{-1}}Z_{d}=\diag(\overline\zeta_{2^{d+1}},\zeta_{2^{d+1}})=\overline Z_{d}.$$
Thus, $\sigma_{-1}$ acts on $\alpha_d$ by conjugating each entry: ${}^{\sigma_{-1}}\alpha_d=\overline\alpha_d$. Since $JZ_{d}J^{-1}=\overline{Z}_{d}$ and $J^{-1}=-J$, we find that $\gamma_J\alpha_d{\gamma_J}^{-1}={}^{\sigma_{-1}}\alpha_d$ where $\gamma_J=\diag(J,J,\ldots, J).$ Furthermore, one can check that $\gamma_J\beta_d{\gamma_J}^{-1}=\beta_d={}^{\sigma_{-1}}\beta_d$. This confirms that $\gamma_J$ is in the twisted Lefschetz group.

We now consider the action of the Galois element $\sigma_5$, which can be described by $${}^{\sigma_5}Z_{d}=\diag(\zeta_{2^{d+1}}^{5},\overline{\zeta_{2^{d+1}}}{}^{5}).$$ The block entry appearing in the $i^{th}$ diagonal block is $Z_{d}^{2i-1}$. Note that both $2i-1$ and $\langle 5(2i-1)\rangle_{2^{d+1}}$ are odd, positive integers. Thus, the action of $\sigma_5$ can be described by
$${}^{\sigma_5}Z_{d}^{2i-1}=\begin{cases}
Z_{d}^{\langle 5(2i-1)\rangle_{2^{d+1}}} & \text{if } \langle 5(2i-1)\rangle_{2^{d+1}}\leq 2g_d,\\
\overline{Z}_d^{2^{d+1}-\langle 5(2i-1)\rangle_{2^{d+1}}} & \text{if }\langle 5(2i-1)\rangle_{2^{d+1}}> 2g_d.
\end{cases}$$

Based on this, and using the techniques developed in \cite{EmoryGoodson2022} and \cite{GoodsonCatalan}, we find that the $\gamma_d$ defined in the statement of the theorem satisfies $\gamma_d\alpha_d{\gamma_d}^{-1}={}^{\sigma_{5}}\alpha_d$. Furthermore, one can show that $\gamma_d\beta_d{\gamma_d}^{-1}=\beta_d={}^{\sigma_{5}}\beta_d$.

Thus, $\gamma_d$ and $\gamma_J$ are elements of the twisted Lefschetz group. By Proposition \ref{prop:AST=TL} and Theorem \ref{theorem:nondegeneracyx^n-x}, it suffices to show that the subgroup generated by $\gamma_d, \gamma_J$ in the quotient $\ST/\ST^0$ is isomorphic to $\Gal(K/\Q)$, where $K=\Q(\zeta_{2^{d+1}})$ is the field of definition of the endomorphisms. We follow the technique used in \cite[Theorem 4.2]{EmoryGoodson2022}. 

Recall from Lemma \ref{lemma:generatorGalois} that $\Gal(K/\Q)\simeq \langle \sigma_{-1},\sigma_5\rangle$, where the order of $\sigma_5$ in $\Gal(K/\Q)$ is $2^{d-1}$. It is clear that the order of $\gamma_J$ in the quotient group is 2. We now show that $\gamma_d^{2^{d-1}}$ is trivial in the quotient group, but $\gamma_d^b\not \in \ST^0(\Jac(C'_{{2^d+1}}))$ for any positive integer $b<2^{d-1}$.

To simplify the notation, we let $\zeta=\zeta_{2^{d+1}}$ for the remainder of the proof. Since the order of $\sigma_5$ in $\Gal(K/\Q)$ is $2^{d-1}$, the action of $\sigma_5$ on the block matrix $Z_d^j$ satisfies
\begin{align}\label{eqn:sigma_action_Zmatrix}
    (\sigma_5)^{b}(Z_d^j)&=\begin{cases}
Z_d^j & \text{if } b\equiv 0 \pmod{2^{d-1}},\\
Z_d^{jk} & \text{otherwise},
\end{cases}
\end{align}
for some $k\not\equiv 1\pmod{2^{d+1}}$.

We have seen that $\gamma_d \alpha_d \gamma_d^{-1}={}^{\sigma_5}\alpha_d$, and so conjugating $\alpha_d$ by $\gamma_d$ permutes (and sometimes conjugates) the diagonal block entries of $\alpha_d$. Since $\gamma_d \alpha_d \gamma_d^{-1}$ is also a diagonal block matrix, conjugating this by $\gamma_d$ will again just permute (and sometimes conjugate) the diagonal block entries. Hence, $\gamma_d^b\alpha_d\gamma_d^{-b}$ is a diagonal block matrix for any $b$, and we can write $\gamma_d^b\alpha_d\gamma_d^{-b} ={}^{(\sigma_5)^b}\alpha_d$.

By construction, each block entry of $\gamma_d$ and, hence, $\gamma_d^b$ is either $\pm I$ or $\pm J$. The matrix $\gamma_d^b$ has a $\pm J$ block entry if and only if there are some $j$ and $k$ for which ${}^{(\sigma_5)^b}Z_d^j=Z_d^{jk}$ with $Z_d^j\not=Z_d^{jk}$. By Equation \eqref{eqn:sigma_action_Zmatrix}, this is possible if and only if $b$ is not a multiple of $2^{d-1}$. Thus, the order of $\gamma_d$ in the quotient group is $2^{d-1}$.  Note that no power of $\gamma_d$ equals $\gamma_J$ in the quotient group (because of their differing block structures) and $\gamma_J$ and $\gamma_d$ commute.  Thus, the subgroup generated by $\gamma_J, \gamma_d$ in the quotient $\ST/\ST^0$ is isomorphic to $\Gal(K/\Q)\simeq \langle \sigma_{-1},\sigma_5\rangle$.
\end{proof}

\subsection{Sato-Tate groups for curves of the form $y^2=x^{2^m}-1$}

Let $C_{2^m}$ denote the curve $y^2=x^{2^m}-1$ with $m>2$. The genus of $C_{2^m}$ is always odd: $g=2^{m-1}-1$. The main result of this section is Theorem \ref{thm:STx^2m-1} which gives an explicit description of the Sato-Tate group of $\Jac(C_{2^m})$. This is a particularly nice result since we express the generators of  $\ST(\Jac(C_{2^m}))$ in terms of the Sato-Tate groups of the factors of $\Jac(C_{2^m})$.

\subsubsection{Preliminaries}\label{sec:Preliminariesx^2m-1}
Throughout this subsection, let $\zeta_{2^{m}}$ be a primitive $2^{m}$th root of unity. We define $\alpha$ and $\beta$ to be the curve automorphisms
\begin{align*}
    \alpha(x,y)&=(\zeta_{2^{m}} x,y) &
    \beta(x,y)&=\left(\frac{1}{x},\frac{i y}{x^{g+1}}\right).
\end{align*}
It follows from \cite[Table 1]{Muller2022} that these automorphisms generate the automorphism group of $C_{2^m}$. The endomorphism field of $\Jac(C_{2^m})$ is $K=\Q(\zeta_{2^{m}})$ \cite[Theorem 3.5.7]{GalleseGoodsonLombardo}.

In Theorem \ref{thm:nondegeneracyx^2m-1} we proved that the Jacobian variety $\Jac(C_{2^{m}})$ is nondegenerate. This allows us to conclude that 
$\ST(\Jac(C_{2^{m}}))/\ST^0(\Jac(C_{2^{m}}))\simeq \Gal(K/\Q)$. Applying Lemma \ref{lemma:generatorGalois} yields $\Gal(K/\Q)=\langle \sigma_{-1}, \sigma_5\rangle$.

The final preliminary detail we need before computing generators of the Sato-Tate group is a description of certain endomorphisms of $\Jac(C_{2^{m}})$. Recall the notation $\omega_{d,j}$ from Section \ref{sec:MapToLowerDim}: for $1\leq d\leq m-1$ and $0\leq j\leq 2^{d-1}-1$
\begin{align}\label{eq:omegadj}
    \omega_{d,j}=x^{n_{d,j}}\frac{dx}{y},
\end{align}
where $n_{d,j}=2^{m-d-1}(2j+1)-1$. The complex uniformization of $\Jac(C_{2^m})$ can be identified with $H^0((C_{2^m})_\C,\Omega^1_\C)=\C\left\langle\{\omega_{d,j}\}\right\rangle$. We compute pullbacks of the regular 1-forms $\omega_{d,j}$ with respect to the curve automorphisms $\alpha$ and $\beta$ and obtain the following.

\begin{lemma}\label{lemma:pullbacksx^2m-1}
For  $1\leq d\leq m-1$ and $0\leq j\leq 2^{d-1}-1$,    
\begin{align}
    \alpha^*(\omega_{d,j})&=\zeta_{2^{m}}^{n_{d,j}+1}\omega_{d,j}\label{eqn:pullbackalpha}\\
    \beta^*(\omega_{d,j})&=i\omega_{d,j'},\label{eqn:pullbackbeta}
\end{align}
where $j'=2^{d-1}-1-j$.
\end{lemma}
\begin{proof}
    We easily show the validity of Equation \eqref{eqn:pullbackalpha}:
    $$\alpha^*(\omega_{d,j})=\alpha^*({x^{dj}dx}/{y})=\frac{\zeta_{2^{m}}^{n_{d,j}+1}x^{n_{d,j}}dx}{y}=\zeta_{2^{m}}^{n_{d,j}+1}\omega_{d,j}.$$

    To prove that Equation \eqref{eqn:pullbackbeta} holds, first note that 
     $$\beta^*(\omega_{d,j})=\beta^*({x^{n_{d,j}}dx}/{y})=\frac{(1/x)^{n_{d,j}}d(1/x)}{ iy/x^{g+1}}=ix^{g-1-n_{d,j}} dx/y.$$

Furthermore,
$$ g-1-n_{d,j} = 2^{m-1}- 2 - (2^{m-d-1}(2j+1)-1)=2^{m-d-1}(2(2^{d-1}-j-1)+1)-1,$$
which equals $d_{j'}$ when  $j'=2^{d-1}-j-1$. 
Hence, $\beta^*(\omega_{d,j})=ix^{d_{j'}} dx/y,$ which proves the desired result.
\end{proof}

We take the symplectic basis of $H_1(\Jac(C_{2^m})_\C,\C)$ corresponding to the basis $\{\omega_{d,j}\}_{d,j}$ of regular 1-forms  (this basis is symplectic with respect to the skew-symmetric matrix $\diag(J,J,\ldots, J)$) to give explicit descriptions of the associated endomorphisms of $\Jac(C_{2^m})$. The power of $\zeta_{2^m}$ appearing in the pullback of $\omega_{d,j}$ under $\alpha$ satisfies $n_{d,j}+1=2^{m-(d+1)}(2j+1)$. This yields $\zeta_{2^m}^{n_{d,j}+1}=((\zeta_{2^m})^{2^{m-(d+1)}})^{2j+1}$, which equals $(\zeta_{2^{d+1}})^{2j+1}$. These are the powers of $\zeta_{2^{d+1}}$ that appear in the endomorphism $\alpha_d$ in Equation \eqref{eqn:alphad}. Thus, the endomorphism corresponding to $\alpha$ is a block diagonal matrix whose block entries are given by the diagonal matrices $\alpha_d$ for $1\leq d\leq m-1$, written in increasing order on $d$:
\begin{align}\label{eqn:alphaendomorphism}
    \alpha&=\diag(\alpha_1, \alpha_2, \ldots, \alpha_{m-1}).
\end{align}

We compute the endomorphism $\beta$ in a similar manner. We find that the endomorphism $\beta$ is a block diagonal matrix whose block entries are nearly the antidiagonal matrices $\beta_d$. Letting $Z$ be the $(2g\times2g)$-matrix $Z=\diag(i,-i,i,-i,\ldots,i,-i)$, the endomorphism $\beta$ can be written as
\begin{align}\label{eqn:betaendomorphism}
    \beta&=Z\cdot\diag(\beta_1, \beta_2, \ldots, \beta_{m-1}).
\end{align}

\subsubsection{Sato-Tate group results}

\begin{proposition}\label{prop:ST0x^2m-1}
For $m>2$,
$$\ST^0(\Jac(C_{2^m}))\simeq \U(1)\times  \left(\U(1)^{2^{m-2}-1}\right)_2.$$  
\end{proposition}

\begin{proof}
    Using the decomposition of the Jacobians in Equation \eqref{eq:JacobianDecomposition} we have 
    $$\Jac(y^2=x^{2^m}-1) \sim (y^2=x^3-x) \times \displaystyle\prod_{d=2}^{m-1}\Jac(y^2=x^{2^{d}+1}-x).$$
    It is well known that $\ST^0(y^2=x^3-x)\simeq \U(1).$   For the case where $d>1$, by Proposition \ref{prop:ST0x^n-x} we have that the identity component of the Jacobians of each of the factors is 
    \[\ST^0(\Jac(C'_{{2^d+1}}))\simeq\left(\U(1)^{2^{d-2}}\right)_2 .\] Since the Jacobian of $C_{2_m}$ is stably nondegenerate (Theorem \ref{thm:nondegeneracyx^2m-1})$$\ST^0(\Jac(C_{2^m}))\simeq \ST^0(y^2=x^3-x) \times \displaystyle\prod_{d=2}^m \ST^0(\Jac(C'_{{2^d+1}}))$$ and the conclusion follows.
\end{proof}

\theoremST

\begin{proof}
   The proof follows from Theorem \ref{thm:STx^n-x}, the action of the Galois generators $\sigma_{-1}$ and $\sigma_5$, and Equations \eqref{eqn:alphaendomorphism} and \eqref{eqn:betaendomorphism}.
\end{proof}

Examples \ref{example:x16} and \ref{example:x32} in the Introduction make explicit the generators that we obtain from this construction for the curves $y^2=x^{16}-1$ and $y^2=x^{32}-1$.


\section{Moment Statistics}\label{sec:moments}

In this section, we describe the distributions of the coefficients of the characteristic polynomial of  random conjugacy classes in the Sato-Tate groups in Theorems \ref{thm:STx^n-x} and \ref{thm:STx^2m-1}. These moment statistics can be used to verify the equidistribution statement of the generalized Sato-Tate conjecture by comparing them to moment statistics obtained for the coefficients of the normalized $L$-polynomial. The numerical moment statistics are an approximation since one can only ever compute them up to some bound. 

\subsection{Computation Techniques}\label{sec:MomentStatsBackground}
The techniques described in this section are adapted from \cite{EmoryGoodson2022,FiteKedlayaSutherlandMotives2016, GoodsonCatalan}. We define the $n$th moment (centered at 0) of a probability density function to be the expected value of the $n$th power of the values, i.e. $M_n[X]=E[X^n]$.  The moments of a function or distribution can be used to understand the shape of the graph or data. 

We first define the Haar measure on the groups that we obtain for the identity components of the Sato-Tate groups. From Propositions \ref{prop:ST0x^n-x} and \ref{prop:ST0x^2m-1} we see that the possible groups are products of the groups $\U(1)$ and $\U(1)_2.$

We start with the unitary group $\U(1)$ and consider the trace map $\tr$ on a random element $U\in \U(1)$ defined by $z:=\tr(U)=u+\overline{u}=2\cos(\theta)$, where $u=e^{i\theta}$. From here we see that $dz=2\sin(\theta)d\theta$ and 
$$\mu_{\U(1)}= \frac1{2\pi} \frac{dz}{\sqrt{4-z^2}}=\frac1{2\pi} d\theta$$
gives a uniform measure of $\U(1)$ on the eigenangle $\theta\in[-\pi,\pi]$ (see \cite[Section 2]{SutherlandNotes}). We can deduce the following measures
\begin{equation*}\label{eqn:muU1}
    \mu_{\U(1)^n}= \prod_{i=1}^n \frac1{2\pi} \frac{dz_i}{\sqrt{4-z_i^2}}= \prod_{i=1}^n\frac{1}{2\pi}d\theta_i\;\;
\text{ and }\;\;\mu_{(\U(1)_2)^{n}} = \prod_{i=1}^{n} \frac1{2\pi} \frac{dz_i}{\sqrt{4-z_i^2}} = \prod_{i=1}^{n}\frac{1}{2\pi}d\theta_i.
\end{equation*}

Note that though the measure $\mu_{(\U(1)_2)^{n}}$ is expressed the same as the measure $\mu_{\U(1)^{n}}$, we will get a different distribution since in the former case each eigenangle $\theta_i$ occurs with multiplicity 2. These can be generalized even further to give measures for the groups appearing as identity components  in Propositions \ref{prop:ST0x^n-x} and \ref{prop:ST0x^2m-1}.

We now define the moment sequence $M[\mu]$, where $\mu$ is a positive measure on some interval $I=[-d,d]$. The $n^{th}$ moment $M_n[\mu]$  is the expected value of $\phi_n$ with respect to $\mu$, where $\phi_n$ is the function $z\mapsto z^n$. It is therefore given by 
$M_n[\mu] = \int_I z^n\mu(z).$
Computing this for the measures $\mu_{\U(1)}$ and $\mu_{\U(1)_2}$  yields $M_n[\mu_{\U(1)}]  = \binom{n}{n/2}$ and $M_n[\mu_{\U(1)_2}] =2^n\binom{n}{n/2}$, where $\binom{n}{n/2}=0$ if $n$ is odd. We can take binomial convolutions of these moment identities to compute moments for the groups appearing as identity components in Propositions \ref{prop:ST0x^n-x} and \ref{prop:ST0x^2m-1}.  

More generally, let $B$ be a random element (with respect to the normalized Haar measure) of a compact subgroup of $\USp(2g)$ and let $a_i:=a_i(B)$ denote the $i^{th}$ coefficient of the characteristic polynomial of $B$. The $n^{th}$ moment $M_n[a_i]$ is the expected value of $a_i^n$ and we can compute this by integrating against the Haar measure. We obtain moment statistics for the full Sato-Tate group of an abelian variety by taking the average of the moments for $U\cdot B$ for each element $B$ in the component group, where $U$ is a random element of the identity component. We work out an explicit example in Section \ref{sec:momentsgenus4} to demonstrate how these techniques work in practice.

The moment statistics coming from the Sato-Tate group can then be compared to numerical moment statistics coming from the coefficients of the normalized $L$-polynomial of the abelian variety. The numerical moments in Sections \ref{sec:momentsgenus4} and \ref{sec:momentshighergenus} were computed for primes up to $2^{22}$ using an algorithm described in \cite{HarveySuth2014} and \cite{HarveySuth2016}. The bound was chosen to balance two competing interests: accuracy and computation time. While it would be ideal to be able to use a larger bound for the prime $p$, we found that this bound was sufficient for comparison to the moments coming from the Sato-Tate group. On the other hand, this limitation provides extra motivation to compute explicit generators for the Sato-Tate group in order to give an accurate description of the limiting distributions of the coefficients of normalized $L$-polynomials.

\subsection{Genus 4 Example}\label{sec:momentsgenus4}

We begin with the genus 4 curve $C'_9:y^2=x^9-x$. We use Theorem \ref{thm:STx^n-x} 
 to determine that the component group of Sato-Tate group is generated by the block matrices
 $$\begin{pmatrix}0&0&I&0\\J&0&0&0\\0&0&0&-J\\0&I&0&0 \end{pmatrix}$$
  and $\gamma_J=\diag(J,J,J,J)$.
The identity component is the connected group $\left(\U(1)\times \U(1)\right)_2$. Let $U=\diag(u_0,\overline u_0, u_1, \overline u_1,u_0,\overline u_0, u_1, \overline u_1)$ where $u_j=e^{i\theta_j}$, be an element in the identity component. We compute the characteristic polynomial of each matrix $U\gamma^k\gamma_J^m$, where $0\leq k\leq 7$ and $0\leq m\leq 1$. We find that $a_1$-coefficient of the characteristic polynomial $U\gamma^k\gamma_J^m$ equals 0 unless $k=m=0$. The $n^{th}$ moment $M_n[a_1(U)]$ is given by
{$$\frac{2^{n-2}}{\pi^2}\int_0^{2\pi}\int_0^{2\pi} \left( 
\cos\left(\theta_0\right)+\cos\left(\theta_0\right)+
\cos\left(\theta_1\right)+\cos\left(\theta_1\right)
\right)^n d\theta_0\,d\theta_1.$$}
We then compute the moment statistics $M_n[\mu_1]$ of the full Sato-Tate group by averaging over the size of the group.

Table \ref{table:x9x} below gives (rounded) numerical $a_1$-moment statistics and the corresponding exact $\mu_1$-moment statistics coming from the Sato-Tate group. The odd moments $M_n[\mu_1] = 0$ for all odd $n$, so we omit those values from the table. The errors in the even moments are consistent with those for the odd moments and would be improved by choosing a larger bound for $p$. 

\begin{table}[h]
\begin{tabular}{|c|p{2cm}|p{2cm}|p{2cm}|p{2cm}|p{2cm}|}
\hline
\multicolumn{5}{|c|}{$a_1$-moments}\\
\hline
 &$M_2$ & $M_4$ & $M_6$ & $M_8$ \\ 
\hline
$a_1$&  1.989  & 71.299  & 3154.55 & 153942\\
\hline
$\mu_1$ & 2  &  72  &  3200 &  156800 \\
\hline
\end{tabular}
\caption{Table of $a_1$- and $\mu_1$-moments ($p<2^{22}$)  for $y^2=x^9-x$.}\label{table:x9x}
\end{table}

\subsection{Higher Genus Examples}\label{sec:momentshighergenus}

In this section we present moment statistics for some higher genus curves. Computing the numerical moments becomes more difficult as the genus grows and limits the bound we can use on the primes. In order to compute numerical statistics using primes up to $2^{22}$ for the curves $y^2=x^{16}-1$ and $y^2=x^{32}-1$, we relied on the factorizations of their Jacobians and $L$-polynomials. For example, for the genus 7 curve $y^2=x^{16}-1$, we use Equation \eqref{eq:JacobianDecomposition} to write
$$\Jac(y^2=x^{16}-1)\sim \Jac(y^2=x^3-x)\times \Jac(y^2=x^5-x)\times \Jac(y^2=x^9-x).$$
This tells us that the $L$-polynomial of $\Jac(y^2=x^{16}-1)$ (for good primes $p$) factors into the product of $L$-polynomials of the factor curves. The $a_1$-trace  of the normalized $L$-polynomial of $\Jac(y^2=x^{16}-1)$ is therefore the sum of the $a_1$-traces of the factors, which we were able to compute. The numerical and Sato-Tate moments are given in Table \ref{table:x16}.

\begin{table}[h]
\begin{tabular}{|c|p{2cm}|p{2cm}|p{2cm}|p{2cm}|p{2cm}|}
\hline
 &$M_2$ & $M_4$ & $M_6$ & $M_8$ \\ 
\hline
$a_1$& 4.9749 & 238.008 & 20277.7 & 2203730\\
\hline
$\mu_1$ & 5 & 243 &  21170 & 2358755\\
\hline
\end{tabular}
\caption{Table of $a_1$- and $\mu_1$-moments ($p<2^{22}$) for $y^2=x^{16}-1$.}\label{table:x16}
\end{table}

We were able to directly compute moment statistics for the genus 8 curve $y^2=x^{17}-x$ for primes up to $2^{22}$. The numerical and Sato-Tate moments for the $a_1$-trace are given in Table \ref{table:x17x}.

\begin{table}[h]
\begin{tabular}{|c|p{2cm}|p{2cm}|p{2cm}|p{2cm}|p{2cm}|}
\hline
 &$M_2$ & $M_4$ & $M_6$ & $M_8$ \\ 
\hline
$a_1$ & 1.98265 & 164.813 & 19727.9 & 2861530\\
\hline
$\mu_1$ & 2 & 168 & 20480 & 3041920\\
\hline
\end{tabular}
\caption{Table of $a_1$- and $\mu_1$-moments ($p<2^{22}$) for $y^2=x^{17}-x$.}\label{table:x17x}
\end{table}

For the genus 15 curve $y^2=x^{32}-1$, we once again made use of the factorization of the Jacobian and its $L$-polynomial. From Equation \eqref{eq:JacobianDecomposition} we have
$$\Jac(y^2=x^{32}-1)\sim \Jac(y^2=x^3-x)\times \Jac(y^2=x^5-x)\times \Jac(y^2=x^9-x)\times \Jac(y^2=x^{17}-x)$$
and so we add the $a_1$-traces of the factors to obtain the $a_1$-trace for $y^2=x^{32}-1$. The numerical and Sato-Tate moments are given in Table \ref{table:x32}.

\begin{table}[h]
\begin{tabular}{|c|p{2cm}|p{2cm}|p{2cm}|p{2cm}|p{2cm}|}
\hline
 &$M_2$ & $M_4$ & $M_6$ & $M_8$ \\ 
\hline
$a_1$ &6.94177 & 698.749 & 150219 & 46647200\\
\hline
$\mu_1$ & 7 &  723 & 159190 &  49909475\\
\hline
\end{tabular}
\caption{Table of $a_1$- and $\mu_1$-moments ($p<2^{22}$) for $y^2=x^{32}-1$.}\label{table:x32}
\end{table}
\section {Declarations}
\begin{itemize}
\item Conflict of interest/Competing interest: The authors hereby declare that the disclosed information is correct and that no other situation of real, potential or apparent conflict of interest is known to us. In particular, there are not financial or non-financial interests that are directly or indirectly related to the work submitted for publication.

\item Data Availability Statement: We do not analyze or generate any datasets, because our work proceeds within a theoretical and mathematical approach. Code for calculations in Sage can be found at  https://github.com/heidigoodson/Nondegeneracy-and-Sato-Tate-Distributions-of-Two-Families-of-Jacobian-Varieties. 
\end{itemize}

\bibliographystyle{abbrv}
\bibliography{SatoTatebib}

\end{document}